\documentclass[10pt]{amsart}
\usepackage{fancyhdr}
\usepackage{stmaryrd}
\usepackage{calrsfs}

\usepackage{amsmath}
\usepackage[nospace,noadjust]{cite}
\usepackage{amsfonts}
\usepackage{amssymb,enumerate}
\usepackage{amsthm}
\usepackage{cite}
\usepackage{comment}
\usepackage{color}
\usepackage[all]{xy}
\usepackage{hyperref}
\usepackage{lineno}

\newtheorem*{maintheorem*}{Main Theorem}
\newtheorem{theorem}{Theorem}[section]
\newtheorem{prop}[theorem]{Proposition}

\newtheorem{cor}[theorem]{Corollary}
\theoremstyle{definition}
\newtheorem{definition}[theorem]{Definition}
\newtheorem{remark}[theorem]{Remark}
\newtheorem{example}[theorem]{Example}
\numberwithin{equation}{section}

\newcommand{\nn}{\mathbb{N}}
\newcommand{\qq}{\mathbb{Q}}
\newcommand{\rr}{\mathbb{R}}
\newcommand{\zz}{\mathbb{Z}}

\newcommand{\cone}{\mathsf{cone}}
\newcommand{\aff}{\mathsf{aff}}

\newcommand{\gp}{\text{gp}}

\newcommand{\rank}{\mathsf{rank}}
\newcommand{\rk}{\text{rank}}

\newcommand{\pl}{\mathcal{L}}
\newcommand{\ps}{\mathcal{S}}
\newcommand{\ii}{\mathcal{A}}
\newcommand{\uu}{\mathcal{U}}

\providecommand\ldb{\llbracket}
\providecommand\rdb{\rrbracket}

\hypersetup{
	pdftitle={OHF},
	colorlinks=true,
	linkcolor=blue,
	citecolor=cyan,
	urlcolor=wine
}



\begin{document}
	
	\mbox{}
	\title{Length-factoriality in commutative monoids and \\ integral domains}
	
	\author{Scott T. Chapman}
	\address{Department of Mathematics and Statistics\\Sam Houston State University\\Huntsville, TX 77341}
	\email{scott.chapman@shsu.edu}
	
	\author{Jim Coykendall}
	\address{Department of Mathematical Sciences\\Clemson University\\Clemson, SC 29634}
	\email{jcoyken@clemson.edu}
	
	\author{Felix Gotti}
	\address{Department of Mathematics\\MIT\\Cambridge, MA 02139}
	\email{fgotti@mit.edu}
	
	\author{William W. Smith}
	\address{Department of Mathematics\\The University of North Carolina\\Chapel Hill, NC 27599}
	\email{wwsmith@email.unc.edu}
	
\date{\today}

\begin{abstract}
	 An atomic monoid $M$ is called a length-factorial monoid (or an other-half-factorial monoid) if for each non-invertible element $x \in M$ no two distinct factorizations of $x$ have the same length. The notion of length-factoriality was introduced by Coykendall and Smith in 2011 as a dual of the well-studied notion of half-factoriality. They proved that in the setting of integral domains, length-factoriality can be taken as an alternative definition of a unique factorization domain. However, being a length-factorial monoid is in general weaker than being a factorial monoid (i.e., a unique factorization monoid). Here we further investigate length-factoriality. First, we offer two characterizations of a length-factorial monoid $M$, and we use such characterizations to describe the set of Betti elements and obtain a formula for the catenary degree of $M$. Then we study the connection between length-factoriality and purely long (resp., purely short) irreducibles, which are irreducible elements that appear in the longer (resp., shorter) part of any unbalanced factorization relation. Finally, we prove that an integral domain cannot contain purely short and a purely long irreducibles simultaneously, and we construct a Dedekind domain containing purely long (resp., purely short) irreducibles but not purely short (resp., purely long) irreducibles.
\end{abstract}
\medskip

\maketitle


\bigskip
\section{Introduction}
\label{sec:intro}

An atomic monoid $M$ is called half-factorial if for all non-invertible $x \in M$, any two factorizations of~$x$ have the same length. In contrast to this, we say that $M$ is length-factorial if for all non-invertible $x \in M$, any two distinct factorizations of $x$ have different lengths. An integral domain is called half-factorial if its multiplicative monoid is half-factorial. Half-factorial monoids and domains have been systematically investigated during the last six decades in connection with algebraic number theory, combinatorics, and commutative algebra: from work that appeared more than two decades ago, such as ~\cite{lC60,aZ80,jC97,jC99}, to more recent literature, including \cite{GKR15,GLTZ19,GZ19,MO16,aP12,mR11,PS20}. The term ``half-factorial" was coined by Zaks in~\cite{aZ80}. On the other hand, length-factorial monoids were first investigated in 2011 by the second and fourth authors~\cite{CS11}. As their main result, they proved that unique factorization domains can be characterized as integral domains whose multiplicative monoids are length-factorial. Recently, length-factorial monoids have been classified in the class of torsion-free rank-$1$ monoids~\cite{fG20a}, in the class of submonoids of finite-rank free monoids~\cite{fG20}, and in the class of monoids of the form $\nn_0[\alpha]$, where $\alpha$ is a positive algebraic numbers~\cite{CG20}. 
\smallskip

Here we offer a deeper investigation of length-factoriality in atomic monoids and integral domains as well as some connections between length-factoriality and the existence of certain extremal irreducible elements, which when introduced were called purely long and purely short irreducibles~\cite{CS11}. We say that a monoid satisfies the PLS property if it contains both purely short and purely long irreducibles. Every length-factorial monoid satisfies the PLS property, and here we determine classes of small-rank monoids where every monoid satisfying the PLS property is length-factorial. We will also establish that the multiplicative monoid of an atomic domain never satisfies the PLS property. As a result, we will rediscover that the multiplicative monoid of an integral domain is length-factorial if and only if the integral domain is a unique factorization domain, which was the main result in~\cite{CS11}.
\smallskip

In Section~\ref{sec:general OHFMs}, which is the first section of content, we offer two characterizations of length-factorial monoids. The first of such characterizations is given in terms of the integral independence of the set of irreducibles and the set of irreducibles somehow shifted. The second characterization states that a non-factorial monoid is length-factorial if and only if the kernel congruence of its factorization homomorphism is nontrivial and can be generated by a single factorization relation. This second characterization will allow us to recover~\cite[Proposition~2.9]{CS11}. In addition, we use the second characterization to determine the set of Betti elements and study the catenary degree of a length-factorial monoid.
\smallskip

In Section~\ref{sec:long and short irreducibles}, we delve into the study of purely long and purely short irreducibles. For an element $x$ of a monoid $M$, a pair of factorizations $(z_1, z_2)$ of $x$ is called irredundant if they have no irreducibles in common and is called unbalanced if $|z_1| \neq |z_2|$. An irreducible $a$ of $M$ is called purely long (resp., purely short) provided that for any pair of irredundant and unbalanced factorizations of the same element, the longer (resp., shorter) factorization contains $a$. We prove that the set of purely long (and purely short) irreducibles of an atomic monoid is finite, and we use this result to decompose any atomic monoid as a direct sum of a half-factorial monoid and a length-factorial monoid.
\smallskip

Section~\ref{sec:finite rank monoids} is devoted to the study of length-factoriality in connection with the PLS property on the class consisting of finite-rank atomic monoids. Observe that this class comprises all finitely generated monoids, all additive submonoids of $\zz^n$, and a large class of Krull monoids. We start by counting the number of non-associated irreducibles of a finite-rank length-factorial monoid. Then we show that for monoids of rank at most $2$, being a length-factorial monoid is equivalent to satisfying the PLS property. We conclude the section by offering further characterizations of length-factoriality for rank-$1$ atomic monoids.

\smallskip
In Section~\ref{sec:integral domains}, we investigate the existence of purely long and purely short irreducibles in the setting of integral domains, arriving to the surprising fact that an integral domain cannot simultaneously contain a purely long irreducible and a purely short irreducible. As a consequence of this fact, we rediscover the main result of~\cite{CS11}, that the multiplicative monoid of an integral domain is length-factorial if and only if the integral domain is a unique factorization domain (a shorter proof of this result was later given in~\cite[Theorem~2.3]{AA10}). We also exhibit examples of Dedekind domains containing purely long (resp., purely short) irreducibles, but not purely short (resp., purely long) irreducibles.

\bigskip
\section{Fundamentals}
\label{sec:fundamentals}

\smallskip
\subsection{General Notation}

Throughout this paper, we let $\mathbb{N}$ denote the set of positive integers, and we set $\nn_0 := \nn \cup \{0\}$. For $a,b \in \zz$ with $a \le b$, we let $\ldb a,b \rdb$ be the discrete interval from $a$ to $b$, that is, $\ldb a, b \rdb = \{n \in \zz : a \le n \le b\}$. In addition, for $S \subseteq \rr$ and $r \in \rr$, we set $S_{\le r} := \{s \in S : s \le r\}$ and, with similar meaning, we use the symbols $S_{\ge r}$, $S_{< r}$, and $S_{> r}$. If $q \in \qq_{> 0}$, then we let $\mathsf{n}(q)$ and $\mathsf{d}(q)$ denote the unique positive integers such that $q = \mathsf{n}(q)/\mathsf{d}(q)$ and $\gcd(\mathsf{n}(q), \mathsf{d}(q)) = 1$. Unless we specify otherwise, when we label elements in a certain set by $s_i, s_{i+1} \dots, s_j$, we always assume that $i,j \in \nn_0$ and that $i \le j$.

\smallskip
\subsection{Commutative Monoids}

We tacitly assume that each monoid (i.e., a semigroup with an identity element) we treat here is cancellative and commutative. As all monoids we shall be dealing with are commutative, we will use additive notation unless otherwise specified. For the rest of this section, let~$M$ be a monoid. We let $M^\bullet$ denote the set $M \! \setminus \! \{0\}$, and we let $\uu(M)$ denote the group consisting of all the units (i.e., invertible elements) of $M$. We say that $M$ is \emph{reduced} if $\uu(M) = \{0\}$.

For the monoid $M$ there exist an abelian group $\text{gp}(M)$ and a monoid homomorphism $\iota \colon M \to \text{gp}(M)$ such that any monoid homomorphism $M \to G$, where $G$ is an abelian group, uniquely factors through~$\iota$. The group $\text{gp}(M)$, which is unique up to isomorphism, is called the \emph{Grothendieck group}\footnote{The Grothendieck group of a monoid is often called the difference or the quotient group depending on whether the monoid is written additively or multiplicatively.} of~$M$. The monoid $M$ is \emph{torsion-free} if $nx = ny$ for some $n \in \nn$ and $x,y \in M$ implies that $x = y$. A monoid is torsion-free if and only if its Grothendieck group is torsion-free (see \cite[Section~2.A]{BG09}). If $M$ is torsion-free, then the \emph{rank} of~$M$, denoted by $\rk(M)$, is the rank of the $\zz$-module $\text{gp}(M)$, that is, the dimension of the $\qq$-vector space $\qq \otimes_\zz \text{gp}(M)$.

An equivalence relation $\rho$ on $M$ is called a \emph{congruence} provided that it is compatible with the operation of $M$, that is, for all $x,y,z \in M$ the inclusion $(y,z) \in \rho$ implies that $(x+y, x+z) \in \rho$. The elements of a congruence are called \emph{relations}. Let $\rho$ be a congruence. Clearly, the set $M/\rho$ of congruence classes (i.e., the equivalence classes) naturally turns into a commutative semigroup with identity (it may not be cancellative). The subset $\{(x,x) : x \in M\}$ of $M \times M$ is the smallest congruence of $M$, and is called the \emph{trivial} (or \emph{diagonal}) congruence. Every relation in the trivial congruence is called \emph{diagonal}, while $(0,0)$ is called the \emph{trivial} relation. We say that $\sigma \subseteq M \times M$ generates the congruence $\rho$ provided that $\rho$ is the smallest (under inclusion) congruence on $M$ containing $\sigma$. A congruence on $M$ is \emph{cyclic} if it can be generated by one element.

For $x,y \in M$, we say that $y$ \emph{divides} $x$ \emph{in} $M$ and write $y \mid_M x$ provided that $x = y + y'$ for some $y' \in M$. If $x \mid_M y$ and $y \mid_M x$, then $x$ and $y$ are said to be \emph{associated} elements (or \emph{associates}) and, in this case, we write $x \sim y$. Being associates determines a congruence on $M$, and $M_{\text{red}} := M/\sim$ is called the \emph{reduced monoid} of $M$. When $M$ is reduced, we identify $M_{\text{red}}$ with~$M$. For $S \subseteq M$, we let $\langle S \rangle$ denote the smallest (under inclusion) submonoid of $M$ containing $S$, and we say that $S$ \emph{generates} $M$ if $M = \langle S \rangle$. An element $a \in M \setminus \uu(M)$ is an \emph{irreducible} (or an \emph{atom}) if for each pair of elements $u,v \in M$ such that $a = u + v$ either $u \in \uu(M)$ or $v \in \uu(M)$. We let $\ii(M)$ denote the set of irreducibles of~$M$. The monoid $M$ is called \emph{atomic} if every element in $M \setminus \uu(M)$ can be written as a sum of atoms. Clearly,~$M$ is atomic if and only if $M_{\text{red}}$ is atomic. Each finitely generated monoid is atomic \cite[Proposition~2.7.8]{GH06}.

\smallskip
\subsection{Factorizations}

The free commutative monoid on the set $\ii(M_{\text{red}})$ is denoted by $\mathsf{Z}(M)$, and the elements of $\mathsf{Z}(M)$ are called \emph{factorizations}. If $z \in \mathsf{Z}(M)$ consists of $\ell$ irreducibles of $M_{\text{red}}$ (counting repetitions), then we call $\ell$ the \emph{length} of $z$ and write $|z| := \ell$. We say that $a \in \ii(M)$ \emph{appears} in $z$ provided that $a + \uu(M)$ is one of the $\ell$ irreducibles of $z$. The unique monoid homomorphism $\pi_M \colon \mathsf{Z}(M) \to M_{\text{red}}$ satisfying $\pi(a) = a$ for all $a \in \ii(M_{\text{red}})$ is called the \emph{factorization homomorphism} of $M$. When there seems to be no risk of ambiguity, we write $\pi$ instead of $\pi_M$. The kernel
\[
	\ker \pi := \{(z,z') \in \mathsf{Z}(M)^2 : \pi(z) = \pi(z') \}
\]
of~$\pi$ is a congruence on $\mathsf{Z}(M)$, which we call the \emph{factorization congruence} of~$M$. In addition, we call an element $(z,z') \in \ker \pi$ a \emph{factorization relation}. Let $(z,z')$ be a factorization relation of $M$. We say that $a \in \ii(M)$ \emph{appears} in $(z,z')$ if $a$ appears in either $z$ or $z'$. We call $(z,z')$ \emph{balanced} if $|z| = |z'|$ and \emph{unbalanced} otherwise. Also, we say that $(z,z')$ is \emph{irredundant} provided that no irreducible of $M$ appears in both $z$ and $z'$. For each $x \in M$ we set
\[
	\mathsf{Z}(x) := \mathsf{Z}_M(x) := \pi^{-1}(x + \uu(M)) \subseteq \mathsf{Z}(M).
\]
Observe that $\mathsf{Z}(u) = \{0\}$ if and only if $u \in \uu(M)$. In addition, note that $M$ is atomic if and only if $\pi$ is surjective, that is $\mathsf{Z}(x) \neq \emptyset$ for all $x \in M$. For each $x \in M$, we set
\[
	\mathsf{L}(x) := \mathsf{L}_M(x) := \{|z| : z \in \mathsf{Z}(x)\} \subset \nn_0.
\]
The monoid $M$ is called a \emph{factorial monoid} (or a \emph{unique factorization monoid}) if $|\mathsf{Z}(x)| = 1$ for all $x \in M$. On the other hand, $M$ is called a \emph{half-factorial monoid} if $|\mathsf{L}(x)| = 1$ for all $x \in M$. Let~$R$ be an integral domain (i.e., a commutative ring with identity and without nonzero zero-divisors). We let $R^\bullet$ denote the multiplicative monoid $R \setminus \{0\}$ and, to simplify notation, we write $\pi_R$ and $\mathsf{Z}(R)$ instead of $\pi_{R^\bullet}$ and $\mathsf{Z}(R^\bullet)$, respectively. In addition, for each $x \in R^\bullet$, we set $\mathsf{Z}_R(x) := \mathsf{Z}_{R^\bullet}(x)$ and $\mathsf{L}_R(x) := \mathsf{L}_{R^\bullet}(x)$. It is clear that $R$ is atomic (resp., a unique factorization domain) if and only if the monoid $R^\bullet$ is atomic (resp., factorial). We say that $R$ is a \emph{half-factorial domain} provided that $R^\bullet$ is a half-factorial monoid. See~\cite{CC00} for a survey on half-factorial domains.

The notion of a half-factorial monoid is therefore obtained from that of a factorial monoid by keeping the existence and weakening the uniqueness of factorizations, i.e., replacing $|\mathsf{Z}(x)| = 1$ by $|\mathsf{L}(x)| = 1$ for every $x \in M$. In~\cite{CS11} the second and fourth authors proposed a dual way to weaken the unique factorization property and obtain a natural relaxed version of a factorial monoid, which they called a length-factorial monoid.

\begin{definition}
	Let $M$ be an atomic monoid. We say that $M$ is \emph{length-factorial} if for all $x \in M$ and $z_1, z_2 \in \mathsf{Z}(x)$ the equality $|z_1| = |z_2|$ implies that $z_1 = z_2$.
\end{definition}

Before proceeding, we make the following observation.

\begin{remark}
	The term ``length-factorial" seems like a natural choice as for every element $x$ of a length-factorial monoid $M$ and every $\ell \in \mathsf{L}(x)$ there is a unique factorization in $\mathsf{Z}(x)$ of length $\ell$. We emphasize, however, that the monoids we study here under the term ``length-factorial monoids" were first investigated in~\cite{CS11} under the term ``other-half-factorial monoids"; observe that the later term highlights the contrast with the half-factorial property.
\end{remark}

Notice that a monoid is length-factorial if and only if its reduced monoid is length-factorial. It is clear that every factorial monoid is a length-factorial monoid. We say that a length-factorial monoid is \emph{proper} if it is not factorial. The study of length-factoriality will be our primary focus of attention here. It has been proved in~\cite{CS11} that the multiplicative monoid of an integral domain is a length-factorial monoid if and only if the integral domain is a unique factorization domain, i.e., the multiplicative monoid of an integral domain cannot be a proper length-factorial monoid. We will obtain this result, along with several additional fundamental results, as a consequence of our investigation.

\bigskip
\section{Characterizations of Length-factorial Monoids}
\label{sec:general OHFMs}

The main purpose of this section is to provide characterizations of a proper length-factorial monoid in terms of the integral dependence of its set of irreducibles and also in terms of its factorization congruence. We will use the established characterizations to describe the set of Betti elements and study the catenary degree of a given length-factorial monoid. Throughout this section, we assume that $M$ is an atomic monoid.

\smallskip
\subsection{Characterizations of a Length-factorial Monoid}

The notion of integral independence plays a central role in our first characterization of a length-factorial monoid. Let $S$ be a subset of $M$. We say that $S$ is \emph{integrally independent} in $M$ if $S$ is  linearly independent as a subset of the $\zz$-module~$\gp(M)$, that is, for any distinct $s_1, \dots, s_n \in S$ and any $c_1, \dots, c_n \in \zz$ the equality $\sum_{i=1}^n c_i s_i = 0$ in $\gp(M)$ implies that $c_i = 0$ for every $i \in \ldb 1,n \rdb$. We proceed to establish two characterizations of proper length-factorial monoids.

\begin{theorem} \label{thm:general characterization of OHFM}
	Let $M$ be an atomic monoid that is not a factorial monoid. Then the following statements are equivalent.
	\begin{enumerate}
		\item[(a)] The monoid $M$ is a length-factorial monoid.
		\smallskip
		
		\item[(b)] There exists $a \in \ii(M_{\emph{red}})$ such that $\ii(M_{\emph{red}}) \setminus \{a\}$ and $a - \ii(M_{\emph{red}}) \setminus \{a\}$ are integrally independent sets in $\emph{\gp}(M_{\emph{red}})$.
		\smallskip
		
		\item[(c)] The congruence $\ker \pi$ is nontrivial and cyclic.
	\end{enumerate}
\end{theorem}

\begin{proof}
	Since $M$ is a length-factorial monoid if and only if $M_{\text{red}}$ is a length-factorial monoid and since the factorization homomorphisms of both $M$ and $M_{\text{red}}$ are the same, there is no loss in assuming that~$M$ is a reduced monoid. Accordingly, we identify $M_{\text{red}}$ with $M$.
	\smallskip
	
	(a) $\Rightarrow$ (b): Assume that $M$ is a length-factorial monoid. Observe that the set $\ii(M)$ cannot be integrally independent as, otherwise, $M$ would be a factorial monoid. Then there exist $a \in \ii(M)$ and $m \in \nn$ such that
	\begin{equation} \label{eq:linearly combination for ma}
		ma = \sum_{i=1}^k m_i a_i
	\end{equation}
	for some $a_1, \dots, a_k \in \ii(M) \setminus \{a\}$ and $m_1, \dots, m_k \in \zz$. Let us verify that $\ii(M) \setminus \{a\}$ is an integrally independent set in $\gp(M)$. Suppose, for the sake of a contradiction, that this is not the case. Then there exist $b \in \ii(M) \setminus \{a\}$ and $n \in \nn$ satisfying
	\begin{equation} \label{eq:linearly combination for nb}
		nb = \sum_{i=1}^\ell n_i b_i
	\end{equation}
	for some $b_1, \dots, b_\ell \in \ii(M) \setminus \{a, b\}$ and $n_1, \dots, n_\ell \in \zz$. Take $c_i = \frac{1}{2}(|m_i| - m_i)$ and $c'_i = \frac{1}{2}(|m_i| + m_i)$ for every $i \in \ldb 1, k \rdb$, and also take $d_i = \frac{1}{2}(|n_i| - n_i)$ and $d'_i = \frac{1}{2}(|n_i| + n_i)$ for every $i \in \ldb 1, \ell \rdb$. Then set
	\[
		z_1 := ma + \sum_{i=1}^k c_i a_i, \quad z_2 := \sum_{i=1}^k c'_i a_i, \quad  w_1 := nb + \sum_{j=1}^\ell d_j b_j, \quad \text{ and } \quad w_2 := \sum_{j=1}^\ell d'_j b_j.
	\]
	It follows from~(\ref{eq:linearly combination for ma}) and~(\ref{eq:linearly combination for nb}) that both $(z_1, z_2)$ and $(w_1, w_2)$ are irredundant factorization relations of~$M$. Because $(z_1, z_2)$ and $(w_1, w_2)$ are irredundant and nontrivial, the length-factoriality of $M$ guarantees that they are both unbalanced. Assume, without loss of generality, that $|z_1| > |z_2|$ and $|w_1| < |w_2|$. Clearly, $((|w_2| - |w_1|)z_1, (|w_2| - |w_1|)z_2)$ and $((|z_1| - |z_2|)w_1, (|z_1| - |z_2|)w_2)$ are both factorization relations of~$M$. By adding them, one can produce a new balanced factorization relation with exactly one of its two factorization components involving the irreducible $a$. However, this contradicts that $M$ is a length-factorial monoid. Thus, $\ii(M) \setminus \{a\}$ is integrally independent in $\gp(M)$.
	\smallskip
	
	Let $a \in \ii(M)$ be as in the previous paragraph. We proceed to argue that the set $a - \ii(M) \setminus \{a\}$ is also integrally independent in $\gp(M)$. Take this time $b_1, \dots, b_\ell \in \ii(M) \setminus \{a\}$ and $n_1, \dots, n_\ell \in \zz$ such that $\sum_{i=1}^\ell n_i (b_i - a) = 0$. Then set $d_i = \frac{1}{2} (|n_i| - n_i)$ and $d'_i = \frac{1}{2}(|n_i| + n_i)$ for every $i \in \ldb 1,\ell \rdb$, and consider the factorizations
	\[
		z_1 := \sum_{i=1}^\ell d_i b_i + \bigg( \sum_{i=1}^\ell d'_i \bigg) a \quad \text{ and } \quad z_2 := \sum_{i=1}^\ell d'_i b_i + \bigg( \sum_{i=1}^\ell d_i \bigg) a.
	\]
	The equality $\sum_{i=1}^\ell n_i b_i = \big( \sum_{i=1}^\ell n_i \big) a$ ensures that $(z_1, z_2)$ is a balanced factorization relation. Since~$M$ is a length-factorial monoid, $z_1 = z_2$ and therefore $n_i = d_i - d'_i = 0$ for every $i \in \ldb 1, \ell \rdb$. As a consequence, we can conclude that $a - \ii(M) \setminus \{a\}$ is an integrally independent set in $\gp(M)$.
	\smallskip
	
	(b) $\Rightarrow$ (c): Suppose that there exists $a \in \ii(M)$ such that both $\ii(M) \setminus \{a\}$ and $a - \ii(M) \setminus \{a\}$ are integrally independent sets in $\gp(M)$. Let $S$ be the subgroup of $\gp(M)$ generated by $\ii(M) \setminus \{a\}$. We have seen before that $\ii(M)$ is an integrally dependent set. As a result, the annihilator $\text{Ann}(a + S)$ of $a + S$ in the $\zz$-module $\gp(M)/S$ is not trivial. Since $\text{Ann}(a + S)$ is an additive subgroup of $\zz$, there exists $m \in \nn$ such that $\text{Ann}(a + S) = m \zz$. Then there is an irredundant factorization relation $(w_1,w_2) \in \ker \pi$ such that exactly $m$ copies of $a$ appear in $w_1$ and no copies of $a$ appear in $w_2$.
	
	Let us verify that $(w_1,w_2)$ is unbalanced. Suppose, by way of contradiction, that $|w_1| = |w_2|$. Note that $\pi(w_1) - \pi(w_2) = 0$ in $\gp(M)$ ensures the existence of $a_0, \dots, a_k \in \ii(M)$ (with $a_0 = a$) and $m_0, \dots, m_k \in \zz$ (with $m_0 = m$) such that $\sum_{i=0}^k m_i a_i = 0$. As $|w_1| = |w_2|$, the equality $\sum_{i=0}^k m_i = 0$ holds. As a consequence, one finds that
	\[
		\sum_{i=1}^k m_i (a - a_i) = a \sum_{i=0}^k m_i - \sum_{i=0}^k m_i a_i = 0.
	\]
	This, along with the fact that $m_i \neq 0$ for some $i \in \ldb 1,k \rdb$, contradicts that $a - \ii(M) \setminus \{a\}$ is an integrally independent set. Hence $|w_1| \neq |w_2|$, and so $(w_1, w_2)$ is unbalanced.
	
	We still need to show that $(w_1,w_2)$ generates the congruence $\ker \pi$. Towards this end, take a nontrivial irredundant factorization relation $(z_1, z_2) \in \ker \pi$. As $\ii(M) \setminus \{a\}$ is integrally independent, $a$ must appear in $(z_1, z_2)$. Assume, without loss of generality, that exactly $n$ copies of $a$ appear in $z_1$ for some $n \in \nn$. Then the equality $\pi(z_1) = \pi(z_2)$ ensures that $n \in \text{Ann}(a + S)$, and so $n = km$ for some $k \in \nn$. Then after canceling $na$ in both sides of $\pi(w_1^k z_2) = \pi(w_2^k z_1)$, we obtain two integral combinations of irreducibles in $\ii(M) \setminus \{a\}$, whose corresponding coefficients must be equal. Thus, $(z_1, z_2) = (w_1, w_2)^k$.
	\smallskip
	
	(c) $\Rightarrow$ (a): Suppose that $\ker \pi$ is a cyclic congruence generated by an unbalanced irredundant factorization relation $(w_1, w_2)$. Let $*$ denote the monoid operation of the congruence $\ker \pi$. Take $(z, z') \in \ker \pi$ such that $z \neq z'$. Since $(w_1, w_2)$ generates $\ker \pi$, there exist $n \in \nn$ and $z_0, \dots, z_n \in \mathsf{Z}(M)$ with $z_0 = z$ and $z_n = z'$ such that for every $i \in \ldb 1,n \rdb$ the equality
	\begin{equation} \label{eq:congruence identities}
		(z_{i-1}, z_i) = (w_1, w_2) \ast (d_i, d_i)
	\end{equation}
	holds for some $d_i \in \mathsf{Z}(M)$. After multiplying all the identities in~(\ref{eq:congruence identities}) (for every $i \in \ldb 1,n \rdb$), one finds that $(z, z') \ast (z_1 \cdots z_{n-1}, z_1 \cdots z_{n-1}) = (w_1^n, w_2^n) \ast (d,d)$, where $d = d_1 \cdots d_n$. Since $z_1 \cdots z_{n-1}$ divides both $w_1^n d$ and $w_2^n d$ in the free monoid $\mathsf{Z}(M)$ and $\gcd(w_1^n, w_2^n) = 1$, there exists $z'' \in \mathsf{Z}(M)$ such that $z_1 \cdots z_{n-1} z'' = d$. As a result, $(z, z') = (z'' w_1^n, z'' w_2^n)$ and so $(z, z')$ is an unbalanced factorization relation. Hence $M$ is a length-factorial monoid.
\end{proof}

Following~\cite{CS11}, we call a factorization relation $(w_1,w_2)$ in $\ker \pi_M$ \emph{master} if any irredundant and unbalanced factorization relation of $M$ has the form $(w_1^n, w_2^n)$ or $(w_2^n, w_1^n)$ for some $n \in \nn$. A master factorization relation must be irredundant and unbalanced unless $M$ is a half-factorial monoid. When~$M$ is a proper length-factorial monoid we have seen that $\ker \pi$ is a nontrivial cyclic congruence, and it is clear that $(w_1, w_2)$ is a generator of $\ker \pi$ if and only if $(w_1, w_2)$ is a master factorization relation, in which case, the only master factorization relations of~$M$ are $(w_1, w_2)$ and $(w_2, w_1)$. In this case, one can readily verify that if $|w_1| < |w_2|$, then $|w_1| < |z| < |w_2|$ for each factorization $z \in \mathsf{Z}(\pi(w_1)) \setminus \{w_1, w_2\}$. As a consequence of Theorem~\ref{thm:general characterization of OHFM}, we obtain the following corollary, which was first established in the proof of the main theorem of~\cite{CS11}.

\begin{cor} \label{cor:OHFM and MF}
	Let $M$ be an atomic monoid. Then $M$ is a proper length-factorial monoid if and only if it admits an unbalanced master factorization relation $(w_1, w_2)$, in which case the only master factorization relations of $M$ are $(w_1, w_2)$ and $(w_2,  w_1)$.
\end{cor}

The numerical monoids that are proper length-factorial monoids have been characterized in~\cite{CS11} as those having precisely two irreducibles. This was generalized in~\cite[Proposition~4.3]{fG20a}, which states that the additive submonoids of $\qq_{\ge 2}$ that are length-factorial monoids are those generated by two elements. In general, every monoid that can be generated by two elements is a length-factorial monoid.

\begin{cor} \label{cor:embedding dimension 2 implies cyclic kernel}
	Let $M$ be a monoid generated by two elements. Then $\ker \pi$ is cyclic, and $M$ is a length-factorial monoid.
\end{cor}

\begin{proof}
	As $M$ is finitely generated, it is atomic. We can assume, without loss of generality, that~$M$ is reduced. If $M$ is a factorial monoid, then there is nothing to show. Therefore assume that~$M$ is not a factorial monoid. Then there exists a generating set $A$ of $M$ with $|A| = 2$. Because~$M$ is not a factorial monoid, $\ii(M) = A$. As both sets $A \setminus \{a\}$ and $a - A \setminus \{a\}$ are singletons, the corollary follows from Theorem~\ref{thm:general characterization of OHFM}.
\end{proof}

When a monoid cannot be generated by two elements, its factorization congruence may not be cyclic (even if the monoid is finitely generated). The next example illustrates this observation.

\begin{example}
	For $n \in \nn_{\ge 3}$, consider the additive submonoid $M = \{0\} \cup \nn_{\ge n}$ of $\nn_0$. It can be readily verified that $M$ is atomic and $\ii(M) = \ldb n, 2n-1 \rdb$. Since $2(n+1) = n + (n+2)$, it follows that $M$ is not a length-factorial monoid. Then Theorem~\ref{thm:general characterization of OHFM} guarantees that the factorization congruence of $M$ is not cyclic.
\end{example}

\smallskip
\subsection{Connection with the Catenary Degree} 

We call a finite sequence $z_0, z_1, \dots, z_k$ of factorizations in $\mathsf{Z}(M)$ a \emph{chain of factorizations} from $z_0$ to $z_k$ if $\pi(z_0) = \pi(z_1) = \dots = \pi(z_k)$, where $\pi$ is the factorization homomorphism of $M$. Consider the subset $\mathcal{R}$ of $\mathsf{Z}(M)^2$ defined as follows: a pair $(z,z') \in \mathsf{Z}(M)^2$ belongs to $\mathcal{R}$ if there exists a chain of factorizations $z_0, z_1, \dots, z_k$ from $z$ to $z'$ such that $\text{gcd}(z_{i-1}, z_i) \neq 1$ for every $i \in \ldb 1,k \rdb$, where $\text{gcd}(z_{i-1},z_i)$ denotes the greatest common divisor of $z_{i-1}$ and $z_i$ as elements of the free commutative monoid $\mathsf{Z}(M)$. It follows immediately that $\mathcal{R}$ is an equivalence relation on $\mathsf{Z}(M)$ that refines $\ker \pi$. For each $x \in M$, we let $\mathcal{R}_x$ denote the set of equivalence classes of $\mathcal{R}$ inside $\mathsf{Z}(x)$. An element $b \in M$ is called a \emph{Betti element} provided that $|\mathcal{R}_x| \ge 2$. Let $\text{Betti}(M)$ denote the set of Betti elements of $M$. As we proceed to show, every proper length-factorial monoid contains essentially one Betti element.

\begin{prop} \label{prop:Betti elements}
	If $M$ is a proper length-factorial monoid, then $|\emph{Betti}(M_{\emph{red}})| = 1$.
\end{prop}

\begin{proof}
	Since $M$ is a proper length-factorial monoid, Corollary~\ref{cor:OHFM and MF} ensures the existence of a master factorization relation $(w_1, w_2)$. Assume that $|w_1| < |w_2|$. We claim that $b = \pi(w_1)$ is a Betti element. To see this, take $w'_1 \in \mathsf{Z}(b)$ with $w'_1 \neq w_1$. As $w_1$ is the minimum-length factorization of the master relation $(w_1, w_2)$, it follows that $|w_1| < |w'_1|$. Therefore $(w_1, w'_1) = (w w_1^n, w w_2^n)$ for some $w \in \mathsf{Z}(M)$ and $n \in \nn$, which implies that $w=1$ and $n=1$, that is, $w'_1 = w_2$. As a result, $\mathsf{Z}(b) = \{w_1, w_2\}$. This, along with the fact that $(w_1, w_2)$ is irredundant, guarantees that $|\mathcal{R}_b| = 2$. Hence $b \in \text{Betti}(M_{\text{red}})$.
	
	Now take $x \in M_{\text{red}}$ such that $x \neq b$, and let us verify that $x$ cannot be a Betti element of $M_{\text{red}}$. If $|\mathsf{Z}(x)| = 1$, then $|\mathcal{R}_x| = 1$, and so $x \notin \text{Betti}(M_{\text{red}})$. Assume, therefore, that $|\mathsf{Z}(x)| \ge 2$. Take $z, z' \in \mathsf{Z}(x)$ with $z \neq z'$ and suppose, without loss of generality, that $|z| < |z'|$. Then $(z, z') = (w w_1^n, w w_2^n)$ for some $w \in \mathsf{Z}(M)$ and $n \in \nn$. If $w \neq 1$, then $z, z'$ is a chain of factorizations from $z$ to $z'$ such that $\gcd(z,z') \neq 1$. Otherwise, the fact that $x \neq b$ ensures that $n \ge 2$, and after taking $z_i = w_1^{n-i} w_2^i$ for each $i \in \ldb 0,n \rdb$, one can readily see that $z_0, z_1, \dots, z_n$ is a chain of factorizations from $z$ to $z'$ satisfying that $\gcd(z_{i-1}, z_i) \neq 1$ for every $i \in \ldb 1,n \rdb$. Hence $|\mathcal{R}_x| = 1$, and so~$x \notin \text{Betti}(M_{\text{red}})$.
\end{proof}

We will conclude this section studying the (monotone, equal) catenary degree of a length-factorial monoid; we express the (monotone) catenary degree in terms of any of the master factorization relations. The \emph{distance} $\mathsf{d}(z,z')$ between two factorizations $z$ and $z'$ in $\mathsf{Z}(M)$ is defined as follows:
\[
	\mathsf{d}(z,z') := \max \bigg\{ \bigg{|} \, \frac{z}{\gcd(z,z')} \bigg{|}, \, \bigg{|} \frac{z'}{\gcd(z,z')} \bigg{|} \, \bigg\}.
\]
It is routine to verify that $\mathsf{d}$ is indeed a distance function. For $N \in \nn_0$, a chain of factorizations $z_0, z_1, \dots, z_k$ is called an $N$-\emph{chain} from $z_0$ to $z_k$ if $\mathsf{d}(z_{i-1}, z_i) \le N$ for every $i \in \ldb 1,k \rdb$. For $x \in M$, we let $\mathsf{c}(x)$ denote the smallest $N \in \nn_0$ such that for every $z,z' \in \mathsf{Z}(x)$ there exists an $N$-chain of factorizations from $z$ to $z'$; when such an $N$ does not exist, we set $\mathsf{c}(x) = \infty$. The \emph{catenary degree} of $M$, denoted by $\mathsf{c}(M)$, is defined by
\[
	\mathsf{c}(M) := \sup \{\mathsf{c}(x) : x \in M\}.
\]
The notion of catenary degree was introduced by Geroldinger in~\cite{aG91} in the context of Noetherian domains, although the term was coined later in~\cite{aG94}. Since then, several variations of the catenary degree have been investigated.

An $N$-chain $z_0, z_1, \dots, z_k$ of factorizations in $\mathsf{Z}(M)$ is said to be \emph{monotone} if $|z_0| \le |z_1| \le \dots \le |z_k|$ or $|z_0| \ge |z_1| \ge \dots \ge |z_k|$. For $x \in M$, we let $\mathsf{c}_{\text{mon}}(x)$ (resp., $\mathsf{c}_{\text{eq}}(x)$) denote the smallest $N \in \nn_0$ such that for every $z,z' \in \mathsf{Z}(x)$ (resp., $z,z' \in \mathsf{Z}(x)$ with $|z| = |z'|$) there exists a monotone $N$-chain of factorizations from $z$ to $z'$; if such an $N$ does not exist, then we set $\mathsf{c}_{\text{mon}}(x) = \infty$ (resp., $\mathsf{c}_{\text{eq}}(x) = \infty$). In addition, we set
\[
	\mathsf{c}_{\text{mon}}(M) := \sup \{\mathsf{c}_{\text{mon}}(x) : x \in M\} \quad \text{ and } \quad \mathsf{c}_{\text{eq}}(M) := \sup \{\mathsf{c}_{\text{eq}}(x) : x \in M\},
\]
and call them the \emph{monotone catenary degree} and the \emph{equal catenary degree} of $M$, respectively. It is clear from the definition that $\mathsf{c}(x) \le \mathsf{c}_{\text{mon}}(x)$ and $\mathsf{c}_{\text{eq}}(x) \le \mathsf{c}_{\text{mon}}(x)$ for all $x \in M$ and, therefore, $\mathsf{c}(M) \le \mathsf{c}_{\text{mon}}(M)$ and $\mathsf{c}_{\text{eq}}(M) \le \mathsf{c}_{\text{mon}}(M)$. For every $\ell \in \nn_0$ and $x \in M$, set $\mathsf{Z}_\ell(x) := \{z \in \mathsf{Z}(x) : |z| = \ell\}$ and define $\mathsf{c}_{\text{adj}}(x)$ as follows:
\[
	\mathsf{c}_{\text{adj}}(x) := \sup \big\{ \mathsf{d}(\mathsf{Z}_k(x), \mathsf{Z}_\ell(x)) \, : \, k,\ell \in \mathsf{L}(x), \, k < \ell, \text{ and } \ldb k, \ell \rdb \cap \mathsf{L}(x) = \{k,\ell\} \big\},
\]
where $\mathsf{d}(Z_1, Z_2) = \min \{\mathsf{d}(z_1, z_2) : z_1 \in Z_1 \text{ and } z_2 \in Z_2 \}$ for any nonempty subsets $Z_1$ and $Z_2$ of $\mathsf{Z}(M)$. The \emph{adjacent catenary degree} of $M$, denoted by $\mathsf{c}_{\text{adj}}(M)$, is then defined as
\[
	\mathsf{c}_{\text{adj}}(M) := \sup\{\mathsf{c}_{\text{adj}}(x) : x \in M\}.
\]
It is clear that $\mathsf{c}_{\text{mon}}(x) = \max \{ \mathsf{c}_{\text{eq}}(x), \mathsf{c}_{\text{adj}}(x) \}$ for all $x \in M$, and so $\mathsf{c}_{\text{mon}}(M) = \max \{ \mathsf{c}_{\text{eq}}(M), \mathsf{c}_{\text{adj}}(M) \}$.
The notion of monotone catenary degree was introduced by Foroutan in~\cite{aF06}, and it has been fairly studied in past literature (see \cite{GR19} and references therein). In \cite[Section~3]{aP15}, Philipp provides characterizations of the monotone, equal, and adjacent catenary degrees of $M$ in terms of the factorization congruence $\ker \pi$.

\begin{prop}
	Let $M$ be a monoid, and let $(w_1, w_2)$ be a master factorization relation of $M$. Then the following statements hold.
	\begin{enumerate}
		\item The monoid $M$ is length-factorial if and only if $\mathsf{c}_{\emph{eq}}(M) = 0$.
		\smallskip
		
		\item If $M$ is a proper length-factorial monoid, then
		\[
			\mathsf{c}_{\emph{adj}}(M) = \mathsf{c}_{\emph{mon}}(M) = \mathsf{c}(M)  = \max\{|w_1|, |w_2|\}.
		\]
	\end{enumerate}
\end{prop}

\begin{proof}
	(1) For the direct implication, assume that $M$ is a length-factorial monoid. Since $M$ is length-factorial, for every $x \in M$ two factorizations in $\mathsf{Z}(x)$ have the same length if and only if they are equal, which immediately implies that $\mathsf{c}_{\text{eq}}(x) = 0$. Hence $\mathsf{c}_{\text{eq}}(M) = 0$. Conversely, suppose that $\mathsf{c}_{\text{eq}}(M) = 0$. Take $x \in M$, and let $z$ and $z'$ be two factorizations of $x$ such that $|z| = |z'|$. Since $\mathsf{c}_{\text{eq}}(x) \le \mathsf{c}_{\text{eq}}(M) = 0$, it follows that $\mathsf{d}(z,z') = 0$, and so $z = z'$. Thus, distinct factorizations of $x$ must have different lengths. Hence $M$ is a length-factorial monoid.
	\smallskip
	
	(2) Now suppose that $M$ is a proper length-factorial monoid. In order to find the catenary degree of~$M$, it suffices to look at the set $\text{Betti}(M)$: indeed, it follows from \cite[Corollary~9]{aP10} that
	\[
		\mathsf{c}(M) = \sup \{\mu(b) : b \in \text{Betti}(M)\},
	\]
	where $\mu(x) = \sup\{ \min_{z \in \rho} |z| : \rho \in \mathcal{R}_x\}$. By Proposition~\ref{prop:Betti elements}, the monoid $M$ contains only one Betti element $b$ up to associate, and we have seen that $\mathcal{R}_b$ consists of two classes, namely, $\{w_1\}$ and $\{w_2\}$. Thus, $\mathsf{c}(M) = \mu(b) = \max\{|w_1|, |w_2|\}$.
	
	Since $\mathsf{c}_{\text{eq}}(M) = 0$, the equality $\mathsf{c}_{\text{adj}}(M) = \mathsf{c}_{\text{mon}}(M)$ holds. Finally, let us argue that $\mathsf{c}_{\text{mon}}(M) = \mathsf{c}(M)$. If $b \in \text{Betti}(M)$, then $\mathsf{Z}(b) = \{w_1, w_2\}$, as we have seen in the proof of Proposition~\ref{prop:Betti elements}. Clearly, $w_1, w_2$ is a monotone $N$-chain of factorizations from~$w_1$ to $w_2$, where $N = \max\{|w_1|, |w_2|\}$. Thus, $\mathsf{c}_{\text{mon}}(b) \le \max\{|w_1|, |w_2|\} = \mathsf{c}(M)$. Now suppose that $x \in M$ is not a Betti element. If $|\mathsf{Z}(x)| = 1$, then $\mathsf{c}_{\text{mon}}(x) = 0 \le \mathsf{c}(M)$. Suppose, otherwise, that $|\mathsf{Z}(x)| > 1$ and take $z, z' \in \mathsf{Z}(x)$ such that $z \neq z'$. As $M$ is a length-factorial monoid, we can assume that $|z| < |z'|$, so $(z,z') = (w w_1^n, w w_2^n)$ for some $w \in \mathsf{Z}(M)$ and $n \in \nn$. In this case, we can take $z_i := w w_1^{n-i} w_2^i$ for each $i \in \ldb 0, n \rdb$ to obtain an $N$-chain of factorizations from $z$ to $z'$, where $N = \max\{|w_1|, |w_2|\} = \mathsf{c}(M)$. This implies that $\mathsf{c}_{\text{mon}}(x) \le \mathsf{c}(M)$. Hence $\mathsf{c}_{\text{mon}}(M) \le \mathsf{c}(M)$ and, therefore, the equality must hold.
\end{proof}

\bigskip
\section{Pure Irreducibles: The PLS Property}
\label{sec:long and short irreducibles}

In this section, we study the notions of purely long and purely short irreducibles (as introduced in~\cite{CS11}) in connection with length-factoriality. Based on these notions of irreducible elements, we introduce a class of atomic monoids that strictly contains that of length-factorial monoids. We will see that each monoid in this new class naturally decomposes as a sum of a half-factorial monoid and a length-factorial monoid. Throughout this section, we let $M$ be an atomic monoid.

\smallskip
\subsection{Pure Irreducibles}

Let $(z_1,z_2)$ be an unbalanced factorization relation of $M$. Then we call the factorization of bigger (resp., smaller) length between $z_1$ and $z_2$ the \emph{longer} (resp., \emph{shorter}) factorization of $(z_1,z_2)$.

\begin{definition}
	Let $M$ be a monoid, and take $a \in \ii(M_{\text{red}})$. We say that $a$ is \emph{purely long} (resp., \emph{purely short}) if $a$ is not prime and for all irredundant and unbalanced factorization relations $(z_1,z_2)$ of $M$, the fact that $a$ appears in $z_1$ implies that $|z_1| > |z_2|$ (resp., $|z_1| < |z_2|$).
\end{definition}

\begin{remark}
	As by definition a purely long (or short) irreducible is not prime, it must appear in at least one nontrivial irredundant factorization relation of $M$.
\end{remark}

We let $\pl(M)$ (resp., $\ps(M)$) denote the set comprising all purely long (resp., purely short) irreducibles of $M_{\text{red}}$. When $M$ is a proper length-factorial monoid, it follows from Corollary~\ref{cor:OHFM and MF} that both $\pl(M)$ and $\ps(M)$ are nonempty sets. More precisely, if $z_1,z_2 \in \mathsf{Z}(M)$ satisfy $|z_1| < |z_2|$ and $(z_1,z_2)$ is an irredundant factorization relation generating the factorization congruence of a length-factorial monoid $M$, then $\pl(M)$ (resp., $\ps(M)$) consists of all irreducibles that appear in $z_2$ (resp., $z_1$).
\smallskip

We call any element of $\pl(M) \cup \ps(M)$ a \emph{pure} irreducible. As a consequence of the following proposition we will obtain that every atomic monoid contains only finitely many pure irreducibles.

\begin{prop} \label{prop:pure irreducibles show in each unbalanced relation}
	For an atomic monoid $M$, let $a$ be a purely short/long irreducible, and let $(w_1,w_2)$ be an irredundant factorization relation. Then $a$ appears in $(w_1, w_2)$ if and only if $(w_1, w_2)$ is unbalanced.
\end{prop}

\begin{proof}
	To argue the direct implication suppose, by way of contradiction, that $(w_1, w_2)$ is balanced. We also assume, without loss of generality, that $a$ appears in $w_2$. Suppose first that $a \in \pl(M)$, and take an irredundant factorization relation $(z_1, z_2)$ such that $|z_1| > |z_2|$ and $a$ appears in $z_1$. Then we can take $n \in \nn$ large enough such that the number of copies of $a$ that appear in $w_1^n z_1$ is strictly smaller than the number of copies of $a$ that appear in $w_2^n z_2$. Therefore $(w_1^n z_1, w_2^n z_2)$ yields, after cancellations, an irredundant and unbalanced factorization relation whose shorter factorization involves $a$. However, this contradicts that $a$ is purely long. Supposing that $a \in \ps(M)$, one can similarly arrive to another contradiction.
	\smallskip
	
	For the reverse implication, assume that $(w_1, w_2)$ is unbalanced with $|w_1| < |w_2|$. Suppose first that $a \in \pl(M)$. Take an irredundant factorization relation $(z_1, z_2)$ such that $a$ appears in $(z_1, z_2)$. There is no loss in assuming that $a$ appears in $z_1$ and, therefore, that $|z_1| > |z_2|$. Then there exists $n \in \nn$ such that $|w_1^n z_1| < |w_2^n z_2|$. Since $a$ appears in the shorter factorization of $(w_1^n z_1, w_2^n z_2)$, the fact that $a$ is a purely long irreducible guarantees that $a$ also appears in the longer factorization of $(w_1^n z_1, w_2^n z_2)$. Hence~$a$ appears in $w_2$. For $a \in \ps(M)$ the proof is similar.
\end{proof}

\begin{cor} \label{cor:finitely many purely irreducibles}
	For an atomic monoid $M$, both sets $\pl(M)$ and $\ps(M)$ are finite.
\end{cor}

\begin{proof}
	If $M$ is a half-factorial monoid, then both sets $\pl(M)$ and $\ps(M)$ are empty. Otherwise, there must exist an unbalanced factorization relation $(z_1, z_2)$. It follows now from Proposition~\ref{prop:pure irreducibles show in each unbalanced relation} that every pure irreducible of $M$ appears in $(z_1, z_2)$. Hence both sets $\pl(M)$ and $\ps(M)$ must be finite.
\end{proof}
\smallskip

Clearly, atomic monoids having both purely long and purely short irreducibles are natural generalizations of length-factorial monoids, and they will play an important role in the remainder of this paper.

\begin{definition}
	If an atomic monoid $M$ contains both purely long and purely short irreducibles, then we say that $M$ has the \emph{PLS property} or that $M$ is a \emph{PLS monoid}.
\end{definition}

For future reference, we highlight the following immediate corollary of Theorem~\ref{thm:general characterization of OHFM}.

\begin{cor} \label{cor:OHFM are PLSM}
	Every proper length-factorial monoid is a PLS monoid.
\end{cor}

The converse of Corollary~\ref{cor:OHFM are PLSM} does not hold even for finitely generated monoids. For any subset $S$ of~$\rr^d$, we let $\cone(S)$ and $\aff(S)$ denote the cone and the affine space generated by $S$, respectively.

\begin{example} \label{ex:fg rank-3 PLSM monoid that is not OHFM}
	For $a_1 = (0,1,1)$, $a_2 = (0,2,1)$, $a_3 = (1,2,3)$, $a_4 = (2,2,2)$, and $a_5 = (3,2,1)$, consider the submonoid $M = \langle a_i : i \in \ldb 1,5 \rdb \rangle$ of $(\nn_0^3,+)$. Clearly, $M$ is atomic and it is not hard to check that $\ii(M) = \{a_i : i \in \ldb 1,5 \rdb\}$. Let $H$ be the hyperplane described by the equation $y = 2$. Since $a_1 \notin H$ and $a_i \in H$ for every $i \in \ldb 2,5 \rdb$, the irreducible $a_1$ is purely long. Because $\cone(a_1, a_2)$ and $\aff(a_3, a_4, a_5)$ only intersect in the origin, $a_1$ and $a_2$ cannot be in the same part of any irredundant factorization relation of $M$. Thus, if $a_2$ appears in an irredundant factorization relation involving $a_1$, then it must appear in its shorter part. In addition, note that because $a_2 \notin \aff(a_3, a_4, a_5)$, there is no irredundant factorization relation of $M$ involving $a_2$ but not $a_1$. Hence $a_2 \in \pl(M)$, and so $M$ is a PLS monoid. However, it follows from~\cite[Section~5]{fG20} that $M$ is not a length-factorial monoid.
\end{example}

None of the conditions $\pl(M) = \emptyset$ and $\ps(M) = \emptyset$ implies the other one. The following example sheds some light upon this observation.

\begin{example} \label{ex:monoids with shorts but no longs}
	For the set $A = \{(0,3), (1,2), (2,1), (3,0)\}$, consider the submonoid $M$ of $(\nn_0^2,+)$ generated by $A$. It is clear that $M$ is atomic, and one can readily check that $\ii(M) = A$. Since all the irreducibles of~$M$ lie in the line determined by the equation $x+y=3$, it follows from \cite[Corollary~5.5]{fG20} that $M$ is a half-factorial monoid.
	
	Now consider the submonoid $M_1$ of $(\nn_0^2,+)$ generated by the set $A_1 = A \cup \{(1,1)\}$. It is easy to verify that $M_1$ is atomic with $\ii(M_1) = A_1$. Moreover, since the irreducibles of $M_1$ are not colinear, it follows from \cite[Corollary~5.5]{fG20} that $M_1$ is not a half-factorial monoid. Therefore there exists an irredundant factorization relation $(z_1, z_2)$ with $|z_1| \neq |z_2|$. Since $M$ is a half-factorial monoid, $(1,1)$ must appear in $(z_1, z_2)$; say that $(1,1)$ appears in $z_1$. After projecting on the line determined by the equation $y = x$, one can easily see that $|z_1| > |z_2|$. As a result, $(1,1)$ is purely long. Note that the irreducibles in $A$ are neither purely long nor purely short because they are precisely the irreducibles of $M$, which is a half-factorial monoid. Hence $M_1$ contains a purely long irreducible but no purely short irreducibles.
	
	Lastly, considering the submonoid $M_2$ of $(\nn_0^2,+)$ generated by the set $A \cup \{(2,2)\}$ and proceeding as we did with $M_1$, one finds that $(2,2)$ is the only purely short irreducible in $M_2$, and also that $M_2$ contains no purely long irreducibles.
\end{example}

We know that half-factorial monoids contain neither purely long nor purely short irreducibles. However, there are monoids that are not half-factorial and still contain neither purely long nor purely short irreducibles.

\begin{example}
	Let $M$ and $A$ be as in Example~\ref{ex:monoids with shorts but no longs}, and let $M_3$ be the submonoid of $(\nn_0^2,+)$ generated by the set $A_3 = A \cup \{ (0,2), (1,1), (2,0) \}$. It is not hard to verify that $M_3$ is an atomic monoid with $\ii(M_3) = A_3$. Since the equalities $2(1,1) = (0,2) + (2,0)$ and $(1,2) + (2,1) = (0,3) + (3,0)$ give rise to two irredundant and balanced factorizations involving each irreducible of $M_3$, the sets $\pl(M_3)$ and $\ps(M_3)$ must be empty. Because of this, $M_3$ cannot be a length-factorial monoid, which is confirmed by \cite[Theorem~5.10]{fG20}. In addition, as the points in $A_3$ are not colinear, it follows from \cite[Corollary~5.5]{fG20} that $M_3$ is not a half-factorial monoid.
\end{example}

\smallskip
\subsection{Sum Decomposition of PLS Monoids}

We proceed to show how to decompose the reduced monoid of a PLS monoid $M$ as the inner sum of a half-factorial monoid $M_1$ and a finitely generated length-factorial monoid $M_2$ satisfying that $M_1 \cap M_2 = \{0\}$. We emphasize that such a decomposition does not guarantee the uniqueness of the representation of an element of $M$ as a sum of an element of~$M_1$ and an element of $M_2$.

\begin{theorem} \label{thm:sum decomposition}
	Let $M$ be a PLS monoid. Then there exist submonoids $H$ and $O$ of $M_{\emph{red}}$ satisfying $M_{\emph{red}} = H + O$, where $H$ is a half-factorial monoid and $O$ is a finitely generated proper length-factorial monoid such that $H \cap O = \{0\}$.
\end{theorem}

\begin{proof}
	Let $O$ be the submonoid of $M_{\text{red}}$ generated by the set $\pl(M) \cup \ps(M)$. It is clear that $O$ is an atomic monoid with $\ii(O) = \pl(M) \cup \ps(M)$. Moreover, note that $\pl(O) = \pl(M)$ and $\ps(O) = \ps(M)$. By Corollary~\ref{cor:finitely many purely irreducibles}, the monoid $O$ is finitely generated. To verify that $O$ is a length-factorial monoid, let $(z_1, z_2)$ be a nontrivial irredundant factorization relation in $\ker \pi_O$. Since at least one irreducible in $\pl(M) \cup \ps(M)$ appears in the relation $(z_1, z_2)$, the latter must be unbalanced by Proposition~\ref{prop:pure irreducibles show in each unbalanced relation}. As a consequence, $O$ is a proper length-factorial monoid.
	
	Now let $H$ be the submonoid of $M_{\text{red}}$ generated by $\ii(M) \setminus (\pl(M) \cup \ps(M))$. It follows immediately that $H$ is atomic with $\ii(H) = \ii(M) \setminus (\pl(M) \cup \ps(M))$. To see that $H$ is a half-factorial monoid, it suffices to observe that since $\ker \pi_H \subseteq \ker \pi_M$, any irredundant factorization relation of $\ker \pi_H$ must be balanced by Proposition~\ref{prop:pure irreducibles show in each unbalanced relation}.
	
	 Because $\ii(M_{\text{red}}) = \ii(H) \cup \ii(O)$, we find that $M_{\text{red}} = H + O$. To argue that $H$ and $O$ have trivial intersection, suppose that $x \in H \cap O$. As both $H$ and $O$ are atomic monoids, one can take $z_1 \in \mathsf{Z}_H(x)$ and $z_2 \in \mathsf{Z}_O(x)$. Therefore $(z_1, z_2) \in \ker \pi_M$. Since $\pl(M) \neq \emptyset$ and $\ps(M) \neq \emptyset$, if a pure irreducible appeared in $z_2$, then a pure irreducible would appear in $z_1$. As $z_1$ consists of non-pure irreducibles, $z_2$ must be the factorization with no irreducibles, whence $x = 0$. As a result, $H \cap O = \{0\}$, which implies that $M_\text{red} = H \oplus O$.
\end{proof}

The converse of Theorem~\ref{thm:sum decomposition} does not hold in general, as the following example indicates.

\begin{example}
	Consider the additive submonoid $M$ of $(\nn_0^2,+)$ generated by the set of lattice points $\{(1,1), (0,3), (1,2), (2,1), (3,0) \}$. We have already seen in the second paragraph of Example~\ref{ex:monoids with shorts but no longs} that $\pl(M) = \{(1,1)\}$ and $\ps(M) = \emptyset$. Therefore $M$ is not a PLS monoid. The submonoid $H = \langle (1,2), (0,3) \rangle$ of~$M$ is clearly a factorial monoid and, in particular, a half-factorial monoid. On the other hand, one can see that the submonoid $O = \langle (1,1), (2,1), (3,0) \rangle$ of $M$ is a proper length-factorial monoid by applying Theorem~\ref{thm:general characterization of OHFM} with $a = (1,1)$. It follows immediately that $M = H \oplus O$ even though $M$ is not a PLS monoid.
\end{example}

We conclude this section with the following proposition.

\begin{prop}
	Let $M$ be a PLS monoid. Then there exists an unbalanced factorization relation $(w_1, w_2) \in \ker \pi$ such that every factorization relation of $\ker \pi$ has the form $(w_1^n h_1, w_2^n h_2)$ for some $n \in \nn_0$ and some balanced factorization relation $(h_1, h_2) \in \ker \pi$.
\end{prop}

\begin{proof}
	Take $a \in \pl(M)$. Set $A = \ii(M) \setminus \{a\}$, and let $S$ be the subgroup of $\gp(M)$ generated by~$A$. Since
	~$a$ appears in an irredundant and unbalanced factorization relation of $M$, there exists $m \in \nn$ such that $\text{Ann}(a + S) = m \zz$, where $\text{Ann}(a+S)$ is the annihilator of $a + S$ in the $\zz$-module $\gp(M)/S$. As $ma \in S$, there is an irredundant factorization relation $(w_1, w_2)$ of $M$ such that exactly $m$ copies of $a$ appear in~$w_1$. It follows from Proposition~\ref{prop:pure irreducibles show in each unbalanced relation} that $|w_1| > |w_2|$. Suppose now that $(z_1, z_2)$ is an irredundant factorization relation of $M$ with $|z_1| > |z_2|$, and let $k \in \nn$ be the number of copies of $a$ appearing in $z_1$. Notice that $k \in \text{Ann}(a + S)$, and therefore $k = nm$ for some $n \in \nn$. Then $(w_1^n z_2, w_2^n z_1) \in \ker \pi$ yields, after cancellations, a factorization relation that does not involve $a$. Thus, such a factorization must be balanced by Proposition~\ref{prop:pure irreducibles show in each unbalanced relation} and cannot involve any pure irreducible. So the number of copies of each irreducible $b$ in $\pl(M)$ (resp., $\ps(M)$) that appear in $z_1$ (resp., $z_2$) equals~$n$ times the number of copies of $b$ that appear in $w_1$ (resp., $w_2$). Hence $(z_1, z_2) = (w_1^n h_1, w_2^n h_2)$, where $h_1, h_2 \in \mathsf{Z}(M)$ involve no pure irreducibles. Clearly, $(h_1, h_2) \in \ker \pi$, and Proposition~\ref{prop:pure irreducibles show in each unbalanced relation} guarantees that $|h_1| = |h_2|$.
\end{proof}

\bigskip
\section{Finite-Rank Monoids}
\label{sec:finite rank monoids}

In this section, we continue studying the OHF and the PLS properties, but we restrict our attention to the class of finite-rank monoids.

\smallskip
\subsection{Number of Irreducibles}

If $M$ is a reduced finite-rank factorial monoid, then it follows from~\cite[Proposition~1.2.3(2)]{GH06} that $|\ii(M)| = \rank(M)$. In parallel with this, the cardinality of $\ii(M)$ in a finite-rank proper length-factorial monoid $M$ can be determined.

\begin{prop} \label{prop:number of atoms of a finite-rank OHFM}
	Let $M$ be a proper length-factorial monoid whose rank is finite. Then the equality $|\ii(M_{\emph{red}})| = \rank(M) + 1$ holds.
\end{prop}

\begin{proof}
	As $\gp(M_{\text{red}}) \cong \gp(M)/\uu(M)$, the monoid $M_{\text{red}}$ has finite rank. Hence one can replace $M$ by $M_{\text{red}}$ and assume that $M$ is reduced. Set $r = \rank(M)$ and then embed $M$ into the $\qq$-vector space $V := \qq \otimes_\zz \gp(M) \cong \qq^r$ via $M \hookrightarrow \gp(M) \to \qq \otimes_\zz \gp(M)$, where the injectivity of the second map follows from the flatness of the $\zz$-module $\qq$. So we can think of $M$ as an additive submonoid of the finite-dimensional vector space $\qq^r$. By Theorem~\ref{thm:general characterization of OHFM}, there exists $a \in \ii(M)$ such that $\ii(M) \setminus \{a\}$ and $a - \ii(M) \setminus \{a\}$ are integrally independent sets in $\gp(M)$. In particular, the sets $\ii(M) \setminus \{a\}$ and $a - \ii(M) \setminus \{a\}$ are linearly independent inside the vector space $V$. Because $M$ is atomic, $\gp(M)$ can be generated by $\ii(M)$ as a $\zz$-module and, therefore, $\ii(M)$ is a generating set of $V$. Since $M$ is a proper length-factorial monoid, the monoid $M$ is not a factorial monoid and, consequently, $\ii(M)$ is a linearly dependent set of $V$. This along with the fact that $\ii(M) \setminus \{a\}$ is linearly independent in $V$ implies that $\ii(M) \setminus \{a\}$ is a basis for $V$. Hence $|\ii(M)| = |\ii(M) \setminus \{a\}| + 1 = r + 1$.
\end{proof}

\begin{cor} \label{cor:finite-rank OHFM are f.g}
	Every finite-rank length-factorial monoid is finitely generated.
\end{cor}

The condition of having finite rank in Corollary~\ref{cor:finite-rank OHFM are f.g} is not superfluous. For instance, consider the additive monoid $M = \langle 2,3 \rangle \oplus \nn_0^\infty$, where $\nn_0^\infty$ is the direct sum of countably many copies of $\nn_0$. Since $\langle 2,3 \rangle$ is a proper length-factorial monoid and $\nn_0^\infty$ is a factorial monoid, $M$ is a proper length-factorial monoid. However, $M$ is not finitely generated because $\rank(M) = \infty$. The converse of Proposition~\ref{prop:number of atoms of a finite-rank OHFM} does not hold in general, as the following example shows.

\begin{example} \label{ex:finite-rank monoids that are not OHFM}
	For every $r \in \nn$, consider the submonoid $M_r$ of $(\nn_0^r, +)$ that is generated by the set $S = \{v_0, re_j : j \in \ldb 1,r \rdb \}$, where $v_0 := \{ e_1 + \dots + e_r \}$. It is not hard to verify that $\ii(M_r) = S$, and so $|\ii(M_r)| = r+1$. Notice that each point in $S$ lies in the hyperplane of $\rr^r$ determined by the equation $x_1 + \dots + x_r = r$. Hence it follows from \cite[Corollary~5.5]{fG20} that $M_r$ is a proper half-factorial monoid. Therefore $M_r$ cannot be a length-factorial monoid.
\end{example}

\smallskip
\subsection{Monoids of Small Rank}

As we have emphasized in Corollary~\ref{cor:OHFM are PLSM}, every proper length-factorial monoid is a PLS monoid. We proceed to show that being a length-factorial monoid is equivalent to being a PSLM in the class of torsion-free monoids with rank at most~$2$.

\begin{theorem} \label{thm:PLSM of rank 2}
	For a torsion-free monoid $M$ with $\rank(M) \le 2$, the following statements are equivalent.
	\begin{enumerate}
		\item[(a)] The monoid $M$ is a proper length-factorial monoid.
		\smallskip
		
		\item[(b)] The monoid $M$ is a PLS monoid.
		\smallskip
		
		\item[(c)] The congruence $\ker \pi$ can be generated by an unbalanced factorization relation.
	\end{enumerate}
\end{theorem}

\begin{proof}
	(a) $\Leftrightarrow$ (c): This is part of Theorem~\ref{thm:general characterization of OHFM}.
	\smallskip
	
	(a) $\Rightarrow$ (b): This is Corollary~\ref{cor:OHFM are PLSM}.
	\smallskip
	
	(b) $\Rightarrow$ (a): Assume that $M$ is a PLS monoid, and suppose for the sake of a contradiction that $M$ is not a proper length-factorial monoid. Since $M$ is finitely generated, it is atomic. As $M$ is not a factorial monoid, $|\ii(M)| \ge 2$. We split the rest of the proof into three cases.
	\smallskip
	
	CASE 1: $|\ii(M)| = 2$. In this case, the factorization congruence $\ker \pi$ is cyclic by Corollary~\ref{cor:embedding dimension 2 implies cyclic kernel}, and the existence of purely long/short irreducibles implies that any generator of $\ker \pi$ must be unbalanced, contradicting that $M$ is not a proper length-factorial monoid.
	\smallskip
	
	CASE 2: $|\ii(M)| = 3$. Take $a_1, a_2, a_3 \in M$ such that $\ii(M) = \{a_1, a_2, a_3 \}$. Assume, without loss of generality, that $a_1 \in \pl(M)$ and $a_2 \in \ps(M)$. Now take an irredundant and balanced factorization relation $(z_1, z_2) \in \ker \pi$. Since $a_1$ and $a_2$ are pure irreducibles, none of them can appear in $(z_1, z_2)$. Therefore only copies of the irreducible $a_3$ appear in both $z_1$ and $z_2$. This implies that $z_1 = z_2$. As $(z_1, z_2)$ was taking to be irredundant, it must be trivial. Hence $M$ is a proper length-factorial monoid, a contradiction.
	\smallskip
	
	CASE 3: $|\ii(M)| \ge 4$. Take $a_0 \in \pl(M)$ and $a_3 \in \ps(M)$, and then take $a_1, a_2 \in \ii(M) \setminus \{a_0, a_3\}$ such that $a_1 \neq a_2$. Since $a_0$ is a purely long irreducible, the submonoid $M' := \langle a_1, a_2, a_3 \rangle$ of $M$ must be a half-factorial monoid. Now take a nontrivial factorization relation $(z_1, z_2) \in \ker \pi_{M'}$. As $a_3$ is a purely short irreducible, it does not appear in $(z_1, z_2)$. Therefore either $(z_1, z_2)$ or $(z_2, z_1)$ equals $(m a_1, m a_2)$ for some $n \in \nn$. Now the fact that $M$ is torsion-free, along with the equality $m a_1 = m a_2$, guarantees that $a_1 = a_2$, which is a contradiction.
\end{proof}

\begin{cor}
	If a torsion-free monoid $M$ is generated by at most three elements, then it is a proper length-factorial monoid if and only if it is a PLS monoid.
\end{cor}

\begin{proof}
	There is no loss in assuming that $M$ is reduced. Clearly, $|\ii(M)| \le 3$. Consider the $\qq$-space $V = \qq \otimes_\zz \gp(M)$, and identify $M$ with its isomorphic copy inside $V$ provided by the embedding $M \hookrightarrow \gp(M) \to \qq \otimes_\zz \gp(M)$. As $M$ is atomic, $\ii(M)$ is a spanning set of~$V$, whence $\dim V \le 3$. If $\dim V = 3$, then $\ii(M)$ is linearly independent over $\qq$, in which case $M$ is the free monoid on $\ii(M)$. In this case, $M$ is neither a proper length-factorial monoid nor a PLS monoid. On the other hand, if $\dim V \le 2$, then $\rank(M) \le 2$ and we are done via Theorem~\ref{thm:PLSM of rank 2}.
\end{proof}

However, for a finitely generated monoid containing four or more irreducibles, the PLS property may not imply the OHF property. This has been illustrated in Example~\ref{ex:fg rank-3 PLSM monoid that is not OHFM}. In the same example, we have seen that the condition of having rank at most $2$ is required in Theorem~\ref{thm:PLSM of rank 2}. On the other hand, the following example indicates that the condition of being torsion-free is also required in the statement of Theorem~\ref{thm:PLSM of rank 2}.

\begin{example}
	Fix $n \in \nn$ such that $n \ge 4$, and consider the submonoid $M := \langle a_k : k \in \ldb 1, n \rdb \rangle$ of the additive group $\zz_{n-2} \times \zz^2$, where $a_1 = (0,0,2)$, $a_2 = (0,0,3)$, and $a_k = (k-3,1,0)$ for every $k \in \ldb 3, n \rdb$. Since $M$ is finitely generated, it must be atomic. In addition, it can be readily verified that $\ii(M) = \{a_k : k \in \ldb 1, n \rdb \}$. Now suppose that $(z_1, z_2)$ is an irredundant and unbalanced factorization relation in $\ker \pi$, and assume that $|z_1| < |z_2|$. Since the second component of both $a_1$ and $a_2$ is $0$ and the second component of $a_3, \dots, a_n$ is $1$, the numbers of irreducibles in $\{a_3, \dots, a_n\}$ that appear in~$z_1$ and in $z_2$ must coincide. A similar observation based on third components shows that $a_1$ appears in~$z_2$ but not in $z_1$ and also that $a_2$ appears in $z_1$ but not in $z_2$. Hence $a_1 \in \pl(M)$ and $a_2 \in \ps(M)$, which implies that $M$ is a PLS monoid. Checking that $M$ is not a length-factorial monoid amounts to observing that the equality $(n-2) a_3 = (0,n-2,0) = (n-2) a_4$ yields an irredundant and balanced nontrivial factorization relation of~$M$.
\end{example}

Now we turn to characterize the PSLMs in the class consisting of all torsion-free rank-$1$ monoids, which have been recently studied under the name \emph{Puiseux monoids}. Puiseux monoids have been studied in connection with commutative algebra~\cite{CG19}, commutative factorization theory~\cite{CGG20a}, and noncommutative factorization theory~\cite{BG20}. An updated survey on the atomic structure of Puiseux monoids is given in~\cite{CGG20}. Notice that a Puiseux monoid is reduced unless it is a group (see~\cite[Section~24]{lF70} and \cite[Theorem 2.9]{rG84}).

\begin{prop}
	Let $M$ be an atomic Puiseux monoid. Then the following statements are equivalent.
	\begin{enumerate}
		\item[(a)] The monoid $M$ is a proper length-factorial monoid.
		\smallskip
		
		\item[(b)] The monoid $M$ is a PLS monoid.
		\smallskip
		
		\item[(c)] Both inclusions $\inf \ii(M) \in \pl(M)$ and $\sup \ii(M) \in \ps(M)$ hold.
		\smallskip
		
		\item[(d)] At least one of the inclusions $\inf \ii(M) \in \pl(M)$ or $\sup \ii(M) \in \ps(M)$ holds.
		\smallskip
		
		\item[(e)] The equality $|\ii(M)| = 2$ holds.
	\end{enumerate}
	If any of the conditions above holds, then $\pl(M)$ and $\ps(M)$ are singletons: $\pl(M) = \{ \inf \ii(M) \}$ and $\ps(M) = \{ \sup \ii(M) \}$.
\end{prop}

\begin{proof}
	(a) $\Rightarrow$ (b): This is Corollary~\ref{cor:OHFM are PLSM}.
	\smallskip
	
	(b) $\Rightarrow$ (c): Suppose that $M$ is a PLS monoid, and take $a_\ell \in \pl(M)$ and $a_s \in \ps(M)$. Now take $a \in M$ such that $a \ne a_\ell$. Clearly, $n := \mathsf{n}(a)\mathsf{n}(a_\ell) \in M$ and, moreover, $z_1 := \mathsf{n}(a) \mathsf{d}(a_\ell) a_\ell$ and $z_2 := \mathsf{n}(a_\ell) \mathsf{d}(a) a$ are two factorizations in $\mathsf{Z}(n)$. Since the factorization relation $(z_1, z_2)$ is irredundant and $a_\ell$ appears in $z_1$, one finds that $|z_1| > |z_2|$. Therefore $\mathsf{n}(a) \mathsf{d}(a_\ell) > \mathsf{n}(a_\ell) \mathsf{d}(a)$, which means that $a > a_\ell$. Then we conclude that $\inf \ii(M) = a_\ell \in \pl(M)$. The equality $\sup \ii(M) = a_s$ can be argued similarly, from which one obtains that $\sup \ii(M) \in \ps(M)$.
	\smallskip
	
	(c) $\Rightarrow$ (d): This is obvious.
	\smallskip
	
	(d) $\Rightarrow$ (e): Assume now that $\inf \ii(M) \in \pl(M)$, and take $a_\ell \in \pl(M)$. Since $M$ is an atomic Puiseux monoid that is not a factorial monoid, it follows that $|\ii(M)| \ge 2$. Suppose, by way of contradiction, that $|\ii(M)| \ge 3$, and take irreducibles $a_1, a_2 \in \ii(M) \setminus \{a_\ell\}$ such that $a_1 \neq a_2$. Consider the element $n := \mathsf{n}(a_1) \mathsf{n}(a_2) \in M$. It is clear that both $z_1 := \mathsf{n}(a_2) \mathsf{d}(a_1) a_1$ and $z_2 := \mathsf{n}(a_1) \mathsf{d}(a_2) a_2$ are factorizations in $\mathsf{Z}(n)$, and they have different lengths because $a_1 \neq a_2$. However, the fact that $a_\ell$ does not appear in either $z_1$ or $z_2$ contradicts that $a_\ell \in \pl(M)$. As a result, $|\ii(M)| = 2$. One can similarly obtain $|\ii(M)| = 2$ assuming that $\sup \ii(M) \in \ps(M)$.
	\smallskip
	
	(e) $\Rightarrow$ (a): If $|\ii(M)| = 2$, it follows from Corollary~\ref{cor:embedding dimension 2 implies cyclic kernel} that $M$ is a length-factorial monoid. Taking $a_1$ and $a_2$ to be the two irreducibles of $M$, one finds that $\mathsf{n}(a_2) \mathsf{d}(a_1) a_1$ and $\mathsf{n}(a_1) \mathsf{d}(a_2) a_2$ are two different factorizations of $\mathsf{n}(a_1) \mathsf{n}(a_2) \in M$, and so $M$ is not a factorial monoid. Hence $M$ must be a proper length-factorial monoid.
\end{proof}

\begin{cor}
	Let $N$ be a numerical monoid. Then $\pl(N) \cup \ps(N)$ is nonempty if and only if $|\ii(N)| =~2$, in which case $\pl(N) = \{ \min \ii(N) \}$ and $\ps(N) = \{ \max \ii(N)\}$.
\end{cor}

\bigskip
\section{Pure Irreducibles in Integral Domains}
\label{sec:integral domains}

We proceed to study the existence of purely long and purely short irreducibles in the context of integral domains. Throughout this section, we set $\pl(R) := \pl(R^\bullet)$ and $\ps(R) := \ps(R^\bullet)$ for any atomic integral domain $R$. In addition, when $R$ is a Dedekind domain, we let $\text{Cl}(R)$ denote the divisor class group of~$R$.

\smallskip
\subsection{Examples of Dedekind Domains}

For a finite-rank monoid $M$, we have already seen in Example~\ref{ex:monoids with shorts but no longs} that none of the conditions $\pl(M) = \emptyset$ and $\ps(M) = \emptyset$ implies the other one. In this subsection, we construct examples of Dedekind domains to illustrate that a similar statement holds in the context of atomic integral domains.

The celebrated Claborn's class group realization theorem \cite[Theorem~7]{lC66} states that for every abelian group $G$ there exists a Dedekind domain $D$ such that $\text{Cl}(D) \cong G$. The following refinement of this result, due to Gilmer, Heinzer, and the fourth author, will be crucial in our constructions.

\begin{theorem} \emph{\cite[Theorem~8]{GHS96}} \label{thm:distribution of prime ideals in DD}
	Let $G$ be a countably generated abelian group generated by $B \cup C$ with $B \cap C = \emptyset$ such that $B^* \cup C$ generates $G$ as a monoid for each cofinite subset $B^* \!$ of $B$. Then there exists a Dedekind domain $D$ with class group $G$ satisfying the following conditions:
	\begin{enumerate}
		\item the set $B \cup C$ consists of the classes of $G$ containing nonzero prime ideals;
		\smallskip
		
		\item the set $C$ consists of the classes of $G$ containing infinitely many nonzero prime ideals;
		\smallskip
		
		\item the set $B$ consists of the classes of $G$ containing finitely many nonzero prime ideals.
	\end{enumerate}
	Moreover, the number of prime ideals in each class contained in $B$ can be specified arbitrarily.
\end{theorem}

Further refinements of Claborn's theorem in the direction of Theorem~\ref{thm:distribution of prime ideals in DD} were given by Grams~\cite{aG74}, Michel and Steffan~\cite{MS86}, and Skula~\cite{lS76}. We are in a position now to exhibit a Dedekind domain with a purely short (resp., long) irreducible but no purely long (resp., short) irreducibles.

\begin{example} \label{ex:ID with purely short}
	In this example, we will produce a Dedekind domain in which there is a purely short irreducible but no purely long irreducibles. Toward this end, consider a Dedekind domain $D$ with class group $\text{Cl}(D) \cong \mathbb{Z}/3\mathbb{Z}$. Since $\text{Cl}(D)$ is finite and the class of $\text{Cl}(D)$ corresponding to $2 + 3 \zz$ generates $\text{Cl}(D)$ as a monoid, one can invoke Theorem~\ref{thm:distribution of prime ideals in DD} to assume that the non-principal prime ideals of $D$ distribute within $\text{Cl}(D)$ in the following way. There is a unique nonzero prime ideal $P$ in the class of $\text{Cl}(D)$ corresponding to $1 + 3\zz$ and infinitely many nonzero prime ideals in the class corresponding to $2 + 3\zz$. Observe that we can separate non-prime irreducible principal ideals $I$ of $D$ into the following three types, according to their (unique) factorizations into prime ideals:
	\[
		(1) \ \ I = P^3, \quad \ (2) \ \ I = PQ, \quad \text{ or }  \quad (3) \ \ I = Q_1Q_2Q_3,
	\]
	where $Q, Q_1, Q_2$, and $Q_3$ are prime ideals in the class of $\text{Cl}(D)$ corresponding to $2 + 3\zz$. By construction, there is only one irreducible principal ideal of type~(1), namely, $P^3$.
	
	Take $a \in \ii(D)$ such that $(a) = P^3$. Proving that $a \in \ps(D)$ amounts to verifying that $(a)$ is a purely short irreducible in the atomic monoid $M$ consisting of all nonzero principal ideals of $D$. To do so, suppose that $(a)$ appears in an irredundant factorization relation $(z_1, z_2) \in \ker \pi_M$, where
	\[
		z_1 := (P^3)^k \prod_{i=1}^m(PQ_i) \prod_{j=1}^n(Q_{j,1}Q_{j,2}Q_{j,3}) \ \quad \text{and} \ \quad z_2 := \prod_{i=1}^{m'}(PQ^\prime_i) \prod_{j=1}^{n'} (Q^\prime_{j,1}Q^\prime_{j,2}Q^\prime_{j,3})
	\]
	for some $k \in \nn$, $m, n, m', n' \in \nn_0$, and ideals $Q_i, Q'_i, Q_{j,i}$, and $Q'_{j,i}$ in the class corresponding to $2 + 3 \zz$. As~$D$ is Dedekind, comparing the numbers of copies of $P$ in all the irreducibles of the ideal factorizations~$z_1$ and $z_2$, one obtains that $3k = m' - m$. Now comparing the numbers of copies of prime ideals in the class $2 + 3 \zz$ in all the irreducibles of the ideal factorizations~$z_1$ and $z_2$, one obtains that $m' - m = 3(n-n')$, and so $n-k = n'$. As a result,
	\[
		|z_1| = k+m+n = (3k+m) + (n-k) - k = m' + n' - k < |z_2|.
	\]
	Therefore $a \in \ps(D)$, as desired.
	
	Let us proceed to argue that $D$ contains no purely long irreducibles. By the previous paragraph, $\pl(D)$ contains no irreducibles generating ideals of type~(1). Take $a_2, a_3 \in \ii(D)$ such that the ideals $(a_2)$ and $(a_3)$ are of type~(2) and type~(3), respectively. Then take prime ideals $Q_1, \dots, Q_5$ in the class of $\text{Cl}(D)$ corresponding to $2 + 3\zz$ such that the equalities $(a_2) = PQ_1$ and $(a_3) = Q_2 Q_3 Q_4$ hold, and $Q_5 \notin \{Q_1, Q_2, Q_3, Q_4\}$. Now consider the ideal factorizations
	\begin{align*}
		z_1 &:= (P^3) (Q_2Q_3Q_4), \\
		z_2 &:= (PQ_2) (PQ_3) (PQ_4), \\
		z_3 &:= (P^3) (PQ_1) (Q_3Q_4Q_5) (Q_5^3), \text{and}\\
		z_4 &:= (PQ_5)^4 (Q_1Q_3Q_4).
	\end{align*}
	Notice that $(z_1, z_2) \in \ker \pi_M$ is irredundant and satisfies $|z_1| < |z_2|$. Because $Q_2 Q_3 Q_4$ appears in~$z_1$, it follows that $a_3 \notin \pl(D)$. On the other hand, $(z_3, z_4) \in \ker \pi_M$ is also irredundant, and it satisfies $|z_3| < |z_4|$. Because $PQ_1$ appears in $z_3$, one finds that $a_2 \notin \pl(D)$. As a result, no irreducible generating an ideal of type~(2) or type~(3) is purely long, whence $\pl(D) = \emptyset$.
\end{example}

To complement Example~\ref{ex:ID with purely short}, we proceed to construct a Dedekind domain having a purely long irreducible but no purely short irreducibles.

\begin{example} \label{ex:ID with purely long}
	Let $D$ be a Dedekind domain with $\text{Cl}(D) \cong \mathbb{Z}$. Since $\text{Cl}(D)$ is countable and the set $\{\pm 1\}$ generates $\text{Cl}(D)$ as a monoid, we can assume in light of Theorem~\ref{thm:distribution of prime ideals in DD} that the non-principal prime ideals of $D$ distribute within $\text{Cl}(D)$ as follows. There is a unique nonzero prime ideal, which we denote by $P$, in the class of $\text{Cl}(D)$ corresponding to $-2$; there is a unique nonzero prime ideal, which we denote by $Q$, in the class of $\text{Cl}(D)$ corresponding to $2$; and there are infinitely many nonzero prime ideals in each of the classes corresponding to $-1$ and $1$. We denote the prime ideals in the class corresponding to $-1$ by (annotated)~$N$ and the prime ideals in the class corresponding to $1$ by (annotated) $M$. Notice that we can separate non-prime irreducible principal ideals $I$ of $D$ into the following four types, according to their (unique) factorizations into prime ideals:
	\[
		(1) \ \ I = PQ, \quad \ (2) \ \ I = P N_1 N_2, \quad \ (3) \ \ I = Q M_1 M_2, \quad \text{ or }  \quad (4) \ \ I = NM,
	\]
	where $N, N_1, N_2$ belong to the class of $\text{Cl}(D)$ corresponding to $-1$ and $M, M_1, M_2$ belong to the class of $\text{Cl}(D)$ corresponding to $1$.
	
	Take $a \in \ii(D)$ such that $(a) = PQ$. As in Example~\ref{ex:ID with purely short}, proving that $a \in \pl(D)$ amounts to showing that the principal ideal $(a)$ is a long irreducible in the atomic monoid $M$ consisting of all nonzero principal ideals of $D$. To do this, suppose that $(a)$ appears in an irredundant ideal factorization relation $(z_1, z_2) \in \ker \pi_M$, and write
	\[
		z_1 := (PQ)^k \prod_{i=1}^m (PN_{i,1}N_{i,2}) \prod_{j=1}^n (Q M_{j,1} M_{j,2}) \prod_{k=1}^t (N_k M_k)
	\]
	and
	\[
		z_2 := \prod_{i=1}^{m'} (PN^\prime_{i,1}N^\prime_{i,2}) \prod_{j=1}^{n'} (QM^\prime_{j,1}M^\prime_{j,2}) \prod_{k=1}^{t'} (N^\prime_k M^\prime_k)
	\]
	for some $k,m,n,t,m',n',t' \in \nn_0$. Since $D$ is Dedekind, after comparing the numbers of copies of the prime ideals $P$ and $Q$ that appear in all irreducibles of $z_1$ and $z_2$, one obtains that $m -m' = n - n' = -k$. In addition, after comparing the numbers of copies of prime ideals in the class corresponding to $-1$ that appear in $z_1$ and $z_2$, one obtains that $t - t'  = -2(m - m') = 2k$. As a result,
	\[
		|z_1| = k + m + n + t = (m' + n' + t') + k + (m - m') + (n - n') + (t - t') = |z_2| + k > |z_2|.
	\]
	Therefore we can conclude that $a \in \pl(D)$.
	
	Finally, let us verify that $D$ contains no purely short irreducibles. Since any irreducible generator of $PQ$ is purely long, $D$ contains no purely short irreducibles of type~(1). Take $a_2 \in \ii(D)$ such that $(a_2)$ has type~(2), and then take prime ideals $N_1, N_2$ in the class of $\text{Cl}(D)$ corresponding to $-1$ such that $(a_2) = P N_1 N_2$. In addition, take distinct prime ideals $N_3$ and $N_4$ in the class of $\text{Cl}(D)$ corresponding to $-1$ such that $N_3, N_4 \notin \{N_1, N_2\}$. Finally, take distinct prime ideals $M_1$ and $M_2$ in the class of $\text{Cl}(D)$ corresponding to $1$. Consider the ideal factorizations
	\[
		z_1 := (PQ)(PN_1N_2)(M_1N_3)(M_2N_4) \quad \text{ and } \quad z_2 := (PN_1N_3)(PN_2N_4)(QM_1M_2).
	\]
	Observe that $(z_1, z_2) \in \ker \pi_M$ is an irredundant factorization relation satisfying $|z_1| > |z_2|$. Since $P N_1 N_2$ appears in $z_1$, it follows that $a_2 \notin \ps(D)$. As a result, no irreducible generating an ideal of type~(2) can be purely short. In a similar manner, one can verify that no irreducible generating an ideal of type~(3) is purely short. Now let $a_4$ be an irreducible of $D$ such that $(a_4)$ is a principal ideal of type~(4). Take $N$ and $M$ in the classes of $\text{Cl}(D)$ corresponding to $-1$ and $1$, respectively, such that $(a_4) = NM$, and then consider the ideal factorizations
	\[
		z_3 := (PQ)(NM)^2 \quad \text{ and } \quad z_4 := (P N^2)(Q M^2).
	\]
	Notice that $(z_3, z_4) \in \ker \pi_M$ is an irredundant factorization relation satisfying $|z_3| > |z_4|$. Since $NM$ appears in $z_3$, it follows that $a_4 \notin \ps(D)$. Therefore none of the irreducibles generating ideals of type~(4) is purely short. As a consequence, $\ps(D) = \emptyset$.
\end{example}

\smallskip
\subsection{Integral Domains Do Not Satisfy the PLS Property}

In Section~\ref{sec:finite rank monoids}, we have proved that in the class of torsion-free monoids with rank at most $2$, being a proper length-factorial monoid and being a PLS monoid are equivalent notions. It was proved in~\cite{CS11} that an integral domain has the length-factorial property only if it is a unique factorization domain. In addition, being a proper length-factorial monoid implies having both purely long and purely short irreducibles. This begs the tantalizing question as to whether there is an atomic integral domain with both purely long and a purely short irreducibles. The Dedekind domains constructed in the previous subsection do not satisfy this property, and this is not a coincidence.

\begin{theorem} \label{thm:domains are not PLSD}
	Let $R$ be an atomic domain. Then either $\pl(R) = \emptyset$ or $\ps(R) = \emptyset$.
\end{theorem}

\begin{proof}
	Suppose, by way of contradiction, that $R$ is an atomic domain such that both $\pl(R)$ nor $\ps(R)$ are nonempty sets. We recall that by Corollary~\ref{cor:finitely many purely irreducibles} both $\pl(R)$ and $\ps(R)$ are finite sets. Set $\ell := |\pl(R)|$ and $s := |\ps(R)|$, and then write
	\[
		\pl(R) =: \{\alpha_1, \alpha_2, \dots, \alpha_\ell\} \quad \text{ and } \quad \ps(R) =: \{ \beta_1, \beta_2, \dots, \beta_s \}.
	\]
	Take $\rho \in \ker \pi_R$ to be an irredundant and unbalanced factorization relation. It follows from Proposition~\ref{prop:pure irreducibles show in each unbalanced relation} that each $\alpha_i$ appears in the longer factorization component of $\rho$ and each~$\beta_j$ appears in the shorter factorization component of $\rho$. Therefore there exist factorizations $z,z' \in \mathsf{Z}(R)$ such that 
	\[
		\rho = (\alpha_1^{a_1} \alpha_2^{a_2} \cdots \alpha_\ell^{a_\ell} z, \, \beta^{b_1}_1 \beta^{b_2}_2 \cdots \beta^{b_s}_s z')
	\]
	for some $a_1, \dots, a_\ell, b_1, \dots, b_s \in \nn$ such that none of the $\alpha_i$'s appears in $z$ and none of the $\beta_j$'s appears in $z'$. In addition, as the factorization is irredundant, none of the $\alpha_i$'s appears in $z'$ and none of the $\beta_j$'s appears in $z$. We now derive contradictions in the following three cases.
	\smallskip
	
	CASE 1: $\ell, s \ge 2$. Consider the element
	\begin{equation} \label{eq:auxiliar element}
		x_1 := \alpha_2^{a_2} \alpha_3^{a_3} \cdots \alpha_\ell^{a_\ell} (\alpha_1^{a_1} - \beta_1^{b_1}) \pi_R(z) \in R. 
	\end{equation}
	We claim that $x_1 \neq 0$. Because $R$ contains no nonzero zero-divisors, verifying that $x_1 \neq 0$ amounts to showing that $\alpha_1^{a_1} - \beta_1^{b_1} \neq 0$. Indeed, this must be the case: if $(\alpha_1^{a_1}, \beta_1^{b_1}) \in \ker \pi_R$, then the fact that $\ell \ge 2$ would force the purely long irreducible $\alpha_2$ to appear in the left factorization component of the relation $(\alpha_1^{a_1}, \beta_1^{b_1})$. Hence both $\alpha_1^{a_1} - \beta_1^{b_1}$ and $x_1$ belong to~$R^\bullet$. On the other hand, it follows from~\eqref{eq:auxiliar element} that $\beta_1$ divides~$x_1$ in $R$. Thus, there exist $w_1 \in \mathsf{Z}(R)$ and $w_2 \in \mathsf{Z}_R(\alpha_1^{a_1} - \beta_1^{b_1})$ such that $\beta_1 w_1 \in \mathsf{Z}_R(x_1)$ and $w := \alpha^{a_2}_2 \alpha^{a_3}_3 \cdots \alpha^{a_\ell}_\ell w_2 z \in \mathsf{Z}_R(x_1)$. Now consider the factorization relation $(\beta_1 w_1,w) \in \ker \pi_R$.
	\smallskip
	
	CASE 1.1: $\beta_1$ does not appear in $w$. Since $\beta_1 \in \ps(R)$, it follows that $|w| > |\beta_1 w_1|$. Now the inclusion $\alpha_1 \in \pl(R)$ implies that $\alpha_1$ must appear in $w$ and, therefore, in $w_2$. Thus, $\alpha_1$ divides $\beta_1^{b_1}$ in $R$. As a result, there is a factorization relation $(\alpha_1 w_1', \beta_1^{b_1}) \in \ker \pi_R$ for some $w'_1 \in \mathsf{Z}(R)$. Clearly, $(\alpha_1 w_1', \beta_1^{b_1})$ is a non-diagonal factorization relation. Since $\beta_1 \in \ps(R)$, th inequality $|\alpha_1 w'_1| > |\beta_1^{b_1}|$ must hold. Therefore $(\alpha_1 w_1', \beta_1^{b_1})$ is an unbalanced factorization relation in which $\beta_2$ does not appear. This contradicts that $\beta_2 \in \ps(R)$.
	\smallskip
	
	CASE 1.2: $\beta_1$ appears in $w$. Because $\beta_1$ does not appear in $z$, we see that $\beta_1$ must appear in $w_2$. This implies that $\beta_1$ divides $\alpha_1^{a_1}$ in $R$. Now we can follow an argument completely analogous to that we just used in CASE 1.1 to obtain the desired contradiction.
	\smallskip
	
	CASE 2: $\{\ell,s\} = \{1,n\}$ for some $n \in \nn_{\ge 2}$. Assume will first assume that $\ell = 1$. 
	Recall that $\rho = (\alpha^{a_1}_1 z , \, \beta^{b_1}_1 \beta^{b_2}_2 \cdots \beta^{b_n}_n z')$. In this case, we also impose the condition that the exponent~$a_1$ is the minimum number of copies of the purely long irreducible $\alpha_1$ that can appear in any irredundant and unbalanced factorization relation in $\ker \pi_R$. Using notation similar to that of CASE~1, we now set
	\begin{equation} \label{eq:auxiliary element 2}
		x_2 := \alpha_1^{a_1-1} (\alpha_1 - \beta_1^{b_1}) \pi_R(z) \in R.
	\end{equation}
	Notice that $x_2 \neq 0$ as otherwise $\alpha_1 = \beta_1^{b_1}$, which is clearly impossible. Since $\ell = 1$, it follows from~\eqref{eq:auxiliary element 2} that $\beta_1$ divides $x_2$ in $R$. Then one can take $w_1 \in \mathsf{Z}(R)$ and $w_2 \in \mathsf{Z}_R(\alpha_1 - \beta_1^{b_1})$ such that $(\beta_1 w_1, \alpha_1^{a_1 - 1} w_2 z) \in \ker \pi_R$. It is clear that $\beta_1$ does not divide $\alpha_1 - \beta_1^{b_1}$, whence $\beta_1$ does not appear in $\alpha_1^{a_1 - 1} w_2 z$, which implies that $|\beta_1 w_1| < |\alpha_1^{a_1 - 1} w_2 z|$. By the minimality of $a_1$, we see that $\alpha_1$ must appear in $w_2$. Thus, $\alpha_1$ must divide $\beta_1^{b_1}$ in $R$. In this case, $(\alpha_1 w_3, \beta_1^{b_1}) \in \ker \pi_R$ for some $w_3 \in \mathsf{Z}(R)$, which is a contradiction because $\beta_2$ does not appear in $\beta_1^{b_1}$. The case when $\ell > 1$ and $s = 1$ follows similarly.
	\smallskip
	
	CASE 3: $\ell = s = 1$. In this case, $\rho = (\alpha_1^{a_1} z, \, \beta_1^{b_1} z')$. We assume that the exponent $a_1$ satisfies the same minimality condition that we imposed in CASE~2. Consider the element
	\begin{equation} \label{eq:auxiliary element 3}
 		x_3 := \alpha_1^{a_1-1} (\alpha_1 - \beta_1) \pi_R(z) \in R.
	\end{equation}
	As $\alpha_1 - \beta_1$ is nonzero and $\beta_1$ divides $x_3$ in $R$, there exist factorizations $w_1 \in \mathsf{Z}(R)$ and $w_2 \in \mathsf{Z}_R(\alpha_1 - \beta_1)$ such that $(\beta_1 w_1, \alpha_1^{a_1 - 1}w_2z) \in \ker \pi_R$. Since $\beta_1$ does not divide $\alpha_1 - \beta_1$ in $R$, it cannot appear in $\alpha_1^{a_1 - 1} w_2 z$ and, therefore, $|\beta_1 w_1| < |\alpha_1^{a_1 - 1} w_2 z|$. This, along with the minimality of $a_1$, implies that $\alpha_1$ appears in $w_2$. However, this contradicts that $\alpha_1$ does not divide $\alpha_1 - \beta_1$ in $R$.
\end{proof}

As a consequence of Theorem~\ref{thm:domains are not PLSD}, we rediscover the main result of~\cite{CS11}.

\begin{cor} \emph{\cite[Theorem~2.10]{CS11}}
	Let $R$ be an integral domain. Then $R$ is a unique factorization domain if and only if $R^\bullet$ is a length-factorial monoid.
\end{cor}

\begin{proof}
	Clearly, if $R$ is a unique factorization domain, then $R^\bullet$ is a factorial monoid and, therefore, a length-factorial monoid. For the reverse implication, suppose that $R^\bullet$ is a length-factorial monoid. By Theorem~\ref{thm:domains are not PLSD}, either $\pl(R)$ is empty or $\ps(R)$ is empty. Therefore~$R^\bullet$ is not a proper length-factorial monoid, and so it is a factorial monoid. Hence $R$ is a unique factorization domain.
\end{proof}

\bigskip
\section*{Acknowledgments}

The authors would like to thank an anonymous referee for providing comments and suggestions that helped improve the quality of the present paper. During this collaboration, the third author was supported by the NSF award DMS-1903069.

\bigskip

\end{document}